\documentclass[12pt,a4paper]{amsart}
\usepackage{amsfonts}
\usepackage{amsthm}
\usepackage{amsmath}
\usepackage{amscd}
\usepackage[latin2]{inputenc}
\usepackage{t1enc}
\usepackage[mathscr]{eucal}
\usepackage{indentfirst}
\usepackage{graphicx}
\usepackage{graphics}
\usepackage{pict2e}
\usepackage{epic}
\numberwithin{equation}{section}
\usepackage[margin=2.9cm]{geometry}
\usepackage{epstopdf}

\theoremstyle{plain}
\newtheorem{Th}{Theorem}[section]
\newtheorem{Lemma}[Th]{Lemma}
\newtheorem{Cor}[Th]{Corollary}
\newtheorem{Prop}[Th]{Proposition}

 \theoremstyle{definition}
\newtheorem{Def}[Th]{Definition}

\newtheorem{Rem}[Th]{Remark}
\newtheorem{?}[Th]{Problem}

\begin{document}

\title[Estimates for quasiintegral points in semigroup orbits]
{ On Quantitative estimates for quasiintegral points in orbits of semigroups of rational maps}

\author[Jorge Mello]{Jorge Mello}

\address{University of New South Wales. mailing adress:\newline School of Mathematics and Statistics
UNSW Sydney 
NSW, 2052
Australia.} 

\email{j.mello@unsw.edu.au}

 \subjclass[2010]{Primary}


\begin{abstract} We give quantitative bounds for the number of quasi-integral points in orbits of semigroups of rational maps under some conditions, generalizing previous work of L. C. Hsia and J. Silverman (2011)
for orbits generated by the iterations of one rational map.
\end{abstract}

\maketitle

\section{Introduction} Let $K$ be a number field, $S$ a finite set of places of $K$, and $\varepsilon >0$.
An element $x \in K$ is said to be \textit{quasi-$(S,\varepsilon)$-integral} if 
\begin{center}
 $\displaystyle\sum_{v \in S} \dfrac{[K_v: \mathbb{Q}_v]}{[K:\mathbb{Q}]} \log (\max \{|x|_v, 1 \}) \geq \varepsilon h([x,1])$,
 \end{center} where $h$ is the  absolute logarithmic height in $\mathbb{P}^1(\overline{\mathbb{Q}})$ and $[x,1] \in \mathbb{P}^1(\overline{\mathbb{Q}})$.
 
 Let $\mathcal{F}= \{ \phi_1,..., \phi_k \} \subset K(x)$ be a finite set of rational functions of degree at least $2$, let $P \in K$ and let 
\begin{center}
$\mathcal{O}_{\mathcal{F}}(P)= \{ \phi_{i_n} \circ ... \circ \phi_{i_1}(P) | n \in \mathbb{N} , i_j =1,...,k \}$
\end{center}
denote the forward orbit of $P$ under the semigroup of functions generated by $\mathcal{F}$.
When $k=1$ and $\phi_1^2= \phi_1 \circ \phi_1 \notin k[z]$, Hsia and Silverman proved [3] that the number of quasi-$(S, \varepsilon)$-integral points in the orbit of a point $P$ with infinite orbit is bounded by a constant depending only on $\phi_1, \hat{h}_{\phi_1}(P), \epsilon, S$, and $[K:\mathbb{Q}]$ ( see Section 2 for the correspondent definitions). We also note that these results, according to (Remark 1, [3]), have some applications as  the existence of quantitative estimates for the size of Zsigmondy sets for such orbits and their primitive divisors, as well for quantitative versions of a dynamical local-global principle in orbits on the projective line. This research was also used to prove finiteness of multiplicatively dependent iterated values by rational functions in [1]. 

In this present paper we generalize this bound for cases of dynamical systems with several rational functions, obtaining, among other results, the following integrality
 result for orbits, that recovers Theorem 2.1 of [7] of J. Silverman using a different approach and under different hypothesis.
 
 \textbf{Theorem A} \textit{  Let $\mathcal{F}= \{ \phi_1,..., \phi_k \} \subset K(x)$ be a set of rational functions of respective degrees $2 \leq d_1\leq ...\leq d_k$ that are not totally ramified at the $\mathcal{F}$-orbit of $\infty$ or that the $\mathcal{F}$-orbit of $\infty$ has no repeated points. Then there is a constant
 $\gamma= \gamma(S,\mathcal{F}, [K:\mathbb{Q}])$ such that for all points $P \in \mathbb{P}^1(K)$ that are not $\Phi$-preperiodic for any sequence $\Phi$ of terms in $\mathcal{F}$,
 the number of $S$-integers in the $\mathcal{F}-$orbit of $P$ is bounded by} 
 \begin{center}
  $\# \{Q \in \mathcal{O}_{\mathcal{F}}(P) ; x(Q) \in R_S \} \leq \dfrac{k^M-1}{k-1}$,  \end{center}
  where $x(Q)$ is the $x$-coordinate of $Q$ and \begin{center} $M=\left\lceil \gamma + \log^+_{d_1}\left(\dfrac{h(\mathcal{F})}{\displaystyle\inf_{\Phi}\hat{h}_{\Phi}(P)}\right)\right\rceil+1$. \end{center}

In Sections 2 and 3 we remind important facts about height functions, distance and dynamics on the projective line. In section 4 we state a quantitative version of Roth's theorem and some facts about the index of ramification. The main results are proved in Section 5, namely, Theorem 5.2 and its Corollaries. 
\section{Canonical Heights} We always assume that $K$ is a fixed number field and $K(z)$ is the field of rational functions over $K$ for the rest of the paper. We identify $K \cup \{ \infty \} = \mathbb{P}^1(K)$ by fixing an affine coordinate $z$ on $\mathbb{P}^1$, so $\alpha \in K$ is equal to $[\alpha, 1] \in \mathbb{P}^1(K)$, and the point at infinity is $[1,0]$. In this way, we assume $z$ is the first left coordinate for points in $\mathbb{P}^1$,
and with respect to this affine coordinate, we identify rational self-maps of $\mathbb{P}^1$ with rational functions in $K(z)$. 
 
 If $P= [x_0,...,x_N] \in \mathbb{P}^N(K)$, the naive logarithmic height is given by 
\begin{center}$h(P)= \sum_{v \in M_K} \dfrac{[K_v: \mathbb{Q}_v]}{[K:\mathbb{Q}]} \log(\max_i |x_i|_v)$, \end{center}
where $M_K$ is the set of places of $K$, $M_K^\infty$ is the set of archimedean (infinite) places of $K$, $M_K^0$ is the set of nonarchimedean (finite) places of $K$, and for each $v \in M_K$, $|.|_v$ denotes the corresponding absolute value on $K$ whose restriction to $\mathbb{Q}$ gives the usual $v$-adic absolute value on $\mathbb{Q}$.
Also, we write $K_v$ for the completion of $K$ with respect to $|.|_v$, and we let $\mathbb{C}_v$ denote the completion of an algebraic closure of $K_v$. To simplify notation, we let $d_v=[K_v:\mathbb{Q}_v]/[K:\mathbb{Q}]$.
Initially, let us recall some theorems on height functions.

\begin{Lemma} \label{main} (Theorem 1.1.1, [5]) There is a way to attach to any projective variety $X$ over $\bar{\mathbb{Q}}$ and any line bundle $L$ on $X$ a function 
\begin{center}
 $h_{L}: X(\bar{\mathbb{Q}}) \rightarrow \mathbb{R}$
 \end{center} with the following properties: 
 
 (i) $h_{L\otimes M}= h_{L} + h_{M} + O(1)$ for any line bundles $L$ and $M$ on $X$, where $O(1)$ is a bounded function for $P$ in $X(\bar{\mathbb{Q}})$.

(ii) If $X= \mathbb{P}^N$ and $L=\mathcal{O}_{\mathbb{P}^N}(1)$, then $h_{\mathcal{O}_{\mathbb{P}^N}(1)}= h + O(1)$.

(iii) If $f: Y \rightarrow X$ is a morphism of projective varieties and $L$ is a line bundle on 

$X$, then $h_{f^*L}=h_{L} \circ f + O(1)$. \newline 
Moreover, the height functions $h_{L}$ are determined up to $O(1)$ by the above three properties.
\end{Lemma}
Recalling that a line bundle in called \textit{ very ample} if it has enough global sections to set up an embedding of the variety into some projective space,
and that is called \textit{ample} if one of its positive powers is very ample, we have 
\begin{Lemma}(Theorem 1.1.2, [5]) Assume $L$ is an ample line bundle of $X$. Let $h_{X,L}$ be a height function corresponding to $L$.

(1) (Northcott's finiteness property) For any real number $c$ and positive integer $D$, the set 
\begin{center}
 $\{x \in X(\bar{\mathbb{Q}}) | [\mathbb{Q}(x): \mathbb{Q}] \leq D, h_{L}(x)\leq c\}$ 
\end{center} is finite.

(2) (positivity) There is a constant $c^{\prime}$ such that $h_{L}(x) \geq c^{\prime}$ for all $x \in X(\bar{\mathbb{Q}})$.

\end{Lemma}

Given a projective variety $X$ over a number field $K$ and $L$ a line bundle on $X$, a height function $h_{X,L}$ corresponding to $L$ is fixed. Let $\mathcal{H}$ be a set of morphisms $f:X \rightarrow X$
over $K$ such that $f^*L \cong L^{\otimes d_f}$ for some integer $d_f \geq 2$. For $f \in \mathcal{H}$, we set
\begin{center}
 $c(f):= \sup_{x \in X(\bar{K})} \left|\dfrac{1}{d_f}h_{L}(f(x)) - h_{L}(x)\right|$.
\end{center}For $\bf{f}$ $= (f_i)_{i=1}^\infty$ a sequence with $f_i \in \mathcal{H}$, i.e, $\bf{f}$ $\in \prod_{i=1}^\infty \mathcal{H}$, we set 
\begin{center}
 $c(\bf{f})$ $:= \sup_{i \geq1} c(f_i) \in \mathbb{R} \cup \{+ \infty \}$.
\end{center}
When $c(\bf{f})$ $ < +\infty$, the sequence is said to be \textit{bounded}. The property of being bounded is independent of the choice of height functions corresponding to $L$.

Let $\mathcal{B}$ be the set of all bounded sequences in $\mathcal{H}$, and for $c>0$, we define
\begin{center}
$\mathcal{B}_c := \{ \bf{f}$ $=(f_i)_{i=1}^\infty \in \mathcal{B} | c(\bf{f})$ $\leq c\}$. 
\end{center}
It is easy to see that if $\mathcal{H}$
 is a finite set of self-maps on a projective space, then any sequence of maps arising from $\mathcal{H}$ belongs to $\mathcal{B}_c$ for some $c$. 
 
 In fact, for $\mathcal{H}=\{ g_1,...g_k \}$, we set 
 \begin{equation} J=\{1,...k\}, W:= \prod_{i=1}^\infty J,
 \text{and}\ \textbf{f}_w := (g_{w_i})_{i=1}^\infty \ \text{for}\ w=(w_i) \in W. \end{equation}
 If $c := \max \{ c(g_1),..., c(g_k) \}$, then $\{\textbf{f}_w$ $| w \in W \} \subset \mathcal{B}_c$.
 
 We also let $S:\prod_{i=1}^\infty \mathcal{H} \rightarrow \prod_{i=1}^\infty \mathcal{H}$ be the \textit{shift map} which sends $\bf{f}$ $=(f_i)_{i=1}^\infty$ to \begin{center}$S(\bf{f})$ $=(f_{i+1})_{i=1}^\infty$.\end{center} 
Then $S$ maps $\mathcal{B}$ into $\mathcal{B}$ and $\mathcal{B}_c$ into $\mathcal{B}_c$ for any $c$.

For $\bf{f}$ $= (f_i)_{i=1}^\infty \in \prod_{i=1}^\infty \mathcal{H}$ and $x \in X(\bar{K})$,
denoting $\bf{f}^{(n)}$ $:=f_n(f_{n-1}(...(f_1(x)))$, the set \begin{center} $\{ x, \textbf{f}^{(1)}(x),  \textbf{f}^{(2)}(x),  \textbf{f}^{(3)}(x),... \}$ 
$ =\{ x, f_1(x), f_2(f_1(x)), f_3(f_2(f_1(x))),... \}$ \end{center} is called the \textit{forward orbit} of $x$ under $\bf{f}$, denoted by
$\mathcal{O}_{\bf{f}} (x)$. The point $x$ is said to be $\bf{f}$-\textit{preperiodic} if $\mathcal{O}_{\bf{f}} (x)$ is finite. If $f=f_1=f_2=....,$ then the forward orbit is the forward orbit under $f$ in the usual sense.

\begin{Lemma} (Theorem 2.3, [4]) Let $X$ be a projective variety over $K$, and $L$ a line bundle on $X$. Let $h_L$ be a height function corresponding to $L$.

(1) There is a unique way to attach to each bounded sequence $\bf{f}$ $=(f_i)_{i=1}^\infty \in \mathcal{B}$ a canonical height function
\begin{center}
 $\hat{h}_{ \bf{f}} : X(\bar{K}) \rightarrow \mathbb{R}$
\end{center} such that

(i) $\sup_{x \in X(\bar{K})} |\hat{h}_{ \bf{f}}(x) - h_L(x)|\leq 2 c(\bf{f})$.

(ii) $\hat{h}_{ S (\bf{f})} \circ f_1= d_{f_1} \hat{h}_{ \bf{f}}$. In particular, $\hat{h}_{ S^n (\bf{f})} \circ f_n \circ ...\circ  f_1=d_{f_n}...d_{f_1} \hat{h}_{L, \bf{f}}$.

 (2) Assume $L$ is ample. Then $\hat{h}_{ \bf{f}}$ satisfies the following properties:

(iii) $\hat{h}_{ \bf{f}}(x) \geq 0$ for all $x \in X(\bar{K})$.

(iv) $\hat{h}_{ \bf{f}}(x)=0$ if and only if $x$ is $\bf{f}$-preperiodic.
\newline We call $\hat{h}_{ \bf{f}}$ a canonical height function (normalized) for $\bf{f}$.
\end{Lemma}

\begin{Lemma} (Corollary 2.4, [4])
 Assume $L$ is an ample line bundle on $X$.
 
 (1) Let $c$ be a nonnegative number, and $D$ a positive integer. Then the set
 \begin{center}
  $\displaystyle\bigcup_{\bf{f} \in \mathcal{B}_c} \{x \in X(\bar{K}) | [K(x):K] \leq D, x$ is $\bf{f}$-preperiodic $ \}$
 \end{center} is finite.
 
 (2) Let $\mathcal{H}=\{ g_1,...g_k \}$, and we set $J$ and $ W$ as in (2.1). Then for any positive integer $D$, the set
 \begin{center}
  $ \{x \in X(\bar{K}) | [K(x):K] \leq D, x$ is $\textbf{f}_w$-preperiodic for some $w \in W \}$
 \end{center} is finite.

\end{Lemma}

Under similar conditions of the previous lemma, namely, $X$ is a projective variety over $K$, $L$ is a line bundle on $X$, $\mathcal{H}=\{ g_1,...g_k \}$,  $g_j^*L \cong L^{\otimes d_{g_j}}$, we have \begin{center}$g_1^*L \otimes... \otimes g_k^*L \cong L^{\otimes(d_{g_1}+...+ d_{g_k})}$.\end{center} Thus $(X,g_1,...,g_k)$ becomes a particular case of what we call a \textit{dynamical eigensystem} for $L$ of degree $d_{g_1}+...+ d_{g_k}$.
 For this, Kawaguchi also proved that 
 \begin{Lemma}(Theorem 1.2.1, [5]) There exists the canonical height function 
 \begin{center}
  $\hat{h}_{ \mathcal{H}} : X(\bar{K}) \rightarrow \mathbb{R}$
 \end{center} for $(X,g_1,...,g_k,L)$ characterized by the following two properties : 
 \newline (i) $\hat{h}_{\mathcal{H}} = h_{\mathcal{H}} + O(1) ;$ 
   \newline(ii) $\sum_{j=1}^k \hat{h}_{\mathcal{H}} \circ g_j = (d_{g_1}+...+ d_{g_k})\hat{h}_{\mathcal{H}}$.
  \end{Lemma}
  \begin{Lemma} (Proposition 3.1, [4])
   Give $J= \{ 1,..., k \}$ the discrete topology (each subset is an open set), and let $\nu$ be the measure on $J$ that assigns mass $\dfrac{d_{g_j}}{d_{g_1}+...+ d_{g_k}}$ to $j \in J$.
   Let $\mu:= \prod_{i=1}^\infty \nu$ be the product measure on $W$. Then we have, for $x \in X(\bar{K})$,
  \begin{center}
   $\hat{h}_{ \{g_1,..., g_k \}}(x)= \displaystyle\int_W \hat{h}_{ \bf{f}_w}(x) d\mu (w)$.
  \end{center}
  In particular, \begin{center}$|\hat{h}_{ \{g_1,..., g_k \}}(x) - h_L(x)| \leq 4c$ \end{center} for all $x \in X(\bar{K})$, where $c =\max \{ c(g_1),..., c(g_k) \}$.

  \end{Lemma}
  
  \section{Distance and dynamics on the projective line}
  
  For each $v \in M_K$, we let $\rho_v$ denote the chordal metric defined on $\mathbb{P}^1(\mathbb{C}_v)$, where we recall that for 
  $[x_1,y_1], [x_2,y_2] \in \mathbb{P}^1(\mathbb{C}_v)$,
  \begin{center} $  
\rho_v([x_1,y_1], [x_2,y_2]) = 
     \begin{cases}
       \ \dfrac{|x_1y_2 - x_2y_1|_v}{\sqrt{|x_1|_v^2 +|y_1|_v^2} \sqrt{|x_2|_v^2 +|y_2|_v^2}}
       \ &\quad\text{if} ~ v \in M_K^\infty, \\
       \
       \
       \ \dfrac{|x_1y_2 - x_2y_1|_v}{\max \{ |x_1|_v,|y_1|_v \} \max \{ |x_2|_v,|y_2|_v \}}
       \ &\quad\text{if} ~ v \in M_K^0. \\

     \end{cases}
 $ \end{center}
 \begin{Def} The \textit{logarithmic chordal metric function} 
 \begin{center}
  $\lambda_v : \mathbb{P}^1(\mathbb{C}_v) \times \mathbb{P}^1(\mathbb{C}_v) \rightarrow \mathbb{R} \cup \{ \infty \}  $
  \end{center} is defined by
  \begin{center}
   $\lambda_v([x_1,y_1], [x_2,y_2]) =- \log \rho_v([x_1,y_1], [x_2,y_2]).$
  \end{center}
  \end{Def}
  
  It is a matter of fact that $\lambda_v$ is a particular choice of an \textit{arithmetic distance function} as defined by Silverman [3], which is
  a local height function $\lambda_{\mathbb{P}^1 \times \mathbb{P}^1, \Delta}$, where $\Delta$ is the diagonal of $\mathbb{P}^1 \times \mathbb{P}^1$.
  The logarithmic chordal metric and the usual metric can relate in the following way.
 \begin{Lemma}(Lemma 3, [3])
  Let $v \in M_K$ and let $\lambda_v$ be the logarithmic chordal metric on $\mathbb{P}^1(\mathbb{C}_v)$. Define $l_v=2$ if $v$ is archimedean, and $l_v=1$ if $v$ is nonarchimedean. 
  Then for $x,y \in \mathbb{C}_v$ the inequality $\lambda_v(x,y) > \lambda_v(y,\infty) +\log l_v $ implies
  \begin{center}
   $ \lambda_v(y,\infty) \leq \lambda_v(x,y) + \log |x-y|_v \leq 2 \lambda_v(x,y) + \log l_v$.
  \end{center}

 \end{Lemma}
 
  Now, let $\mathcal{F}= \{\phi_1,..., \phi_k \}$ be such that each $\phi_j: \mathbb{P}^1 \rightarrow \mathbb{P}^1$ is a rational map of degree $d_j \geq 2$ defined over $K$.
We set $J, W$, and $w$ as in (2.1). In this situation we let \begin{center}$\Phi_w^{(n)}=\phi_{w_n} \circ ... \circ \phi_{w_1}$ \end{center} with $\Phi_w^{(0)}=$Id, 
and also $\mathcal{F}_n :=\{ \Phi_w^{(n)} | w \in W \}$. 

For a point $P \in \mathbb{P}^1$, the $\mathcal{F}$-orbit of $P$ is defined as 
\begin{center}
 $\mathcal{O}_{\mathcal{F}}(P)=\{ \phi(P) | \phi \in \bigcup_{n \geq 1} \mathcal{F}_n \}= \{ \Phi_w^{(n)}(P) | n \geq 0, w \in W \} = \bigcup_{w \in W} \mathcal{O}_{\Phi_w} (P)$.
 \end{center}
 The point $P$ is called \textit{preperiodic for $\mathcal{F}$} if $\mathcal{O}_{\mathcal{F}}(P)$ is finite.
 
 We recall that for $P=[x_0, x_1] \in \mathbb{P}^1(K)$ the height of $P$ is 
 \begin{center}
  $h(P)= \sum_{v \in M_K} d_v \log(\max \{ |x_1|_v, |x_1|_v \}$.
 \end{center} And using the definition of $\lambda_v$, we see that
 \begin{center}
  $h(P)= \sum_{v \in M_K}d_v \lambda_v(P, \infty) + O(1)$.
 \end{center}
 
 For a polynomial $f= \sum a_i z^i$ and an absolute value $v \in M_K$, we define 
  $|f|_v=\max_i \{ |a_i|_v \}$  and \begin{center} $h(f)= \sum_{v \in M_K} d_v \log |f|_v$.
 \end{center}
 Given a rational function $\phi(z)= f(z)/g(z) \in K(z)$ of degree $d$ written in normalized form, let us write $f(z)=\sum_{i \leq d} a_i z^i, g(z)=\sum_{i \leq d} b_i z^i $ with $a_d$ and $b_d$ different from zero,
 and $f$ and $g$ relatively prime in $K[z]$. 
 
 For $v \in M_K$, we set $|\phi|_v= \max \{ |f|_v, |g|_v \}$, and then the height of $\phi$ is defined by
 \begin{center}
  $h(\phi):= \sum_{v \in M_K} d_v \log |\phi|_v$.
 \end{center}
 For $\mathcal{F}= \{ \phi_1,..., \phi_k \}$, we define \begin{center} $h(\mathcal{F}) := \max_i h(\phi_i)$. \end{center}
 
 Also, for any $\phi(z), \psi(z)$ rational functions in $K(z)$, it is a fact, by Proposition 5 (c), [3], that
 \begin{center}
  $h(\phi \circ \psi) \leq h(\phi) + (\deg \phi)h(\psi) + (\deg \phi )(\deg \psi) \log 8$.
 \end{center} Using this one can conclude the following preliminar estimate:
 
 \begin{Prop}
  Let $\mathcal{F}= \{ \phi_1,..., \phi_k \}$ be a finite set of rational functions with $\deg \phi_i= d_i \geq 2$, and $d:= \max_i d_i$. Then for all $n \geq 1$ and $\phi \in \mathcal{F}_n$, we have
  \begin{center}
   $h(\phi) \leq \left(\dfrac{d^n-1}{d-1}\right)h(\mathcal{F}) + d^2\left(\dfrac{d^{n-1}-1}{d-1}\right)\log 8$.
  \end{center}
 \end{Prop}
 \begin{proof}
  For $n=1$ the result is easily true. We assume the it is true for $n$. Let $\phi = \phi_{i_{n+1}} \circ \phi_{i_n} \circ ... \circ \phi_{i_1} \in \mathcal{F}_{n+1}$.
  Then by the previous proposition and the induction hypothesis 
  \begin{align*}
  h(\phi) & \leq h(\phi_{i_{n+1}} \circ \phi_{i_n} \circ ... \circ \phi_{i_2}) + d^n h(\phi_{i_1}) + d^{n+1} \log 8\\
  &\leq \left(\dfrac{d^n-1}{d-1}\right)h(\mathcal{F}) + d^2\left(\dfrac{d^{n-1}-1}{d-1}\right)\log 8 + d^n h(\mathcal{F}) + d^{n+1} \log 8\\
  &\leq \left(\dfrac{d^{n+1}-1}{d-1}\right)h(\mathcal{F}) + d^2\left(\dfrac{d^n-1}{d-1}\right)\log 8,
  \end{align*} and we conclude thus the proof.
  \end{proof}
  
  \begin{Lemma}(Theorem 3.20, [8]) For a rational map $\phi: \mathbb{P}^1 \rightarrow \mathbb{P}^1$ of degree $d \geq 2$ defined over $K$ and $L= \mathcal{O}_{\mathbb{P}^1}(1)$, it is true that
  
  (a) $|h(\phi (P))- dh(P)| \leq c_1h(\phi) + c_2$.
  
  (b) $\hat{h}_{\phi}(P)= \lim_n h(\phi^{(n)}(P))/d^n$.
  
  (c) $|\hat{h}_{\phi}(P) - h(P)| \leq c_3h(\phi) + c_4$.
  
  Where $c_1, c_2, c_3$ and $c_4$ above depend only on $d$.
  \end{Lemma}
  Gathering these facts with Lemma 2.3 and Lemma 2.6, we derive the following:
  \begin{Lemma} Let $\mathcal{F}= \{\phi_1,..., \phi_k \}$ such that each $\phi_j: \mathbb{P}^1 \rightarrow \mathbb{P}^1$ is a rational function of degree $d_j \geq 2$ over $K$.
There are constants $c_1, c_2, c_3$ and $c_4$ depending only on the degrees $d_1,...,d_k$ such that
  \begin{center}
   (i) $ |\hat{h}_{\Phi_{w}}(P) - h(P)| \leq c_1 h(\mathcal{F}) + c_3$ ,
   
   (ii)  $|\hat{h}_{\mathcal{F}}(P) - \hat{h}_{\Phi_{w}}(P)| \leq c_3h(\mathcal{F}) + c_4$
  \end{center} for any $P$ whose corresponding orbits are well defined, and any $w= (w_j)_{j=1}^\infty \in W$.

  \end{Lemma}
  
  \section{A distance estimate and a quantitative version of Roth's Theorem}
We will state two known results that will be needed to prove our main theorems. The first one is a result due to Silverman that gives explicit estimates for the dependence on local heights of points and function.
 
 Let us recall that, for a rational function $f(z)$, $P \neq \infty$ and $f(P) \neq \infty$, the \textit{ramification index} of $f$ at $P$ is defined as the order of $P$ as a zero of the rational function $f(z) - f(P)$, i.e., 
 \begin{center}
  $e_P(f)=$ ord$_P(f(z) - f(P))$.
 \end{center} If $P= \infty$, or $f(P)= \infty$ we change coordinates through a linear fractional transformation $L$, such that $L^{-1}(P) =\beta \neq \infty, L^{-1}(f(L(\beta))) \neq \infty $, and define $e_P(f)=e_{\beta}(L^{-1} \circ f \circ L)$. It will not depend on the choice of $L$. We say that $f$ is \textit{totally ramified} at $P$ if $e_P(f)=\deg f$. It is also an exercise to show that 
 \begin{center}
 $e_P(g \circ f)= e_P(f)e_g(f(P))$
 \end{center} for every $f, g$ rational functions and $P \in K \cup \{ \infty \}$.

 The first result is as following:
 
 \begin{Lemma} (Proposition 7, [3])
  Let $\psi \in K(z)$ be a nontrivial rational function, let $S \subset M_K$ be a finite set of absolute values on $K$, each extended in some way to $\bar{K}$, and let $A, P \in \mathbb{P}^1(K)$. Then 
 \begin{center}
  $\sum_{v \in S} \max\limits_{A^{\prime} \in \psi^{-1}(A)} e_{A^{\prime}}(\psi) d_v \lambda_v (P, A^{\prime}) \geq 
  \sum_{v \in S} d_v \lambda_v (\psi(P),A) - O(h(A) + h(\psi) +1)$,
 \end{center} where the implied constant depends only on the degree of the map $\psi$.
\end{Lemma}

The second result is the following quantitative version of Roth's theorem.

\begin{Lemma} (Theorem 10, [3])
 Let $S$ be a finite subset of $M_K$ that contains all infinite places. We assume that each place in $S$ is extended to $\bar{K}$ in some fashion. Let $s$ be the cardinality of $S$, $\Upsilon$ a finite $G_{\bar{K}/K}$-invariant subset of $\bar{K}$, $\beta$ a map $S \rightarrow \Upsilon$, $\mu >2$, and $M \geq 0$. Then there are constants $r_1$ and $r_2$, depending only on $[K: \mathbb{Q}], \#\Upsilon$ and $\mu$, such that there are at most $4^sr_1$ elements $x \in K$ satisfying both of the following conditions:
 \newline \newline (1) $\sum_{v \in S} d_v \log^+ |x - \beta_v|_v^{-1} \geq \mu h(x) - M$.
 \newline \newline (2) $h(x) \geq r_2 \max\limits_{v \in S} \{h(\beta_v), M, 1 \}$.
\end{Lemma}
For effective bounds on this Theorem, we refer to [2].
To end this section let again be $\mathcal{F}= \{\phi_1,..., \phi_k \}$ such that each $\phi_j: \mathbb{P}^1 \rightarrow \mathbb{P}^1$ is a rational function of degree $d_j \geq 2$ defined over $K$.
For $J, W$, and $w$ as in (2.1), we have $\Phi_w:=(\phi_{w_j})_{j=1}^\infty$ and $\Phi_w^{(n)}=\phi_{w_n} \circ ... \circ \phi_{w_1}$ with $\Phi_w^{(0)}=$Id.

We fix $w$ and denote $\Phi:= \Phi_w, \Phi^{n}:=\Phi_w^{(n)}$ by simplicity.

Now, let $P \in \mathbb{P}^1(K)$ be a point whose $\Phi$-orbit does not have any periodic points within it, namely, that $\Phi^{n}(P) \neq \Phi^{m}(P)$ for all $n \neq m$. Then using well known facts such as the multiplicativity of the ramification index, and the formula 
\begin{center}
 $\sum\limits_P e_P(f) = \deg f -2$ 
\end{center} for rational functions $f$, we can compute that
\begin{align*}
e_P(\Phi^{m})&=e_P(\phi_{w_m}\circ...\circ \phi_{w_1})\\ &=e_P(\phi_{w_1})e_{\phi_{w_1}(P)}(\phi_{w_2})...e_{\phi_{w_{m-1}}(...(\phi_{w_1}(P)))}(\phi_{w_m})\\ &= e_P(\phi_{w_1})e_{\Phi^1(P)}(\phi_{w_2})...e_{\Phi^{m-1}(P)}(\phi_{w_m})
\\ &= e_1e_2...e_m, 
\end{align*}

where we make \begin{equation}
               e_i:= e_{\phi_{w_{i-1}}(...(\phi_{w_1}(P)))}(\phi_{w_i})=e_{\Phi^{i-1}(P)}(\phi_{w_i}).
\end{equation} Therefore

 \begin{align*}e_P(\Phi^{m})=e_1e_2...e_m & \leq \left(\dfrac{e_1 +...+ e_m}{m}\right)^m \\ & = \left(\dfrac{(e_1 -1) +...+ (e_m-1)}{m} + 1\right)^m \\ & \leq \left(\dfrac{\sum_{i \leq k}(2d_i -2)}{m} + 1\right)\\ & \leq e^{\sum_{i \leq k}(2d_i -2)}
 =M= M\left(\dfrac{1}{d_1}d_1\right)^m. \end{align*}
 Hence, generalizing a result for just one function, we have just easily proved the following
 
 \begin{Lemma}
  Let $P \in \mathbb{P}^1(K)$ be a point whose $\Phi$-orbit does not repeat points, then there exist two positive contants $\kappa_1 >0$ and $ 0 < \kappa_2<1$ depending only on the degrees of the functions of $\mathcal{F}= \{\phi_1,..., \phi_k \}$
  such that \begin{center}
             $e_P(\Phi^m) \leq \kappa_1 (\kappa_2d_1)^m$ for any $m \geq 0$.
            \end{center}
\end{Lemma}

  It is possible, under some conditions on the system $\mathcal{F}$, to prove this kind of result for any $P$, replacing $d_1$ by $\deg \Phi^m$ as follows:
  \begin{Lemma}
  Suppose that the $\Phi$-orbit of $P$ does not repeat points or otherwise that no point in its $\Phi$-orbit is totally ramified for any $\phi_j$ in $\mathcal{F}$. Then there exist two positive constants $\kappa_1 >0$ and $ 0 < \kappa_2<1$ depending only on the degrees of the functions belonging to $\mathcal{F}= \{\phi_1,..., \phi_k \}$
  such that \begin{center}
             $e_P(\Phi^m) \leq \kappa_1 \kappa_2^m \deg \Phi^m$ for any $m \geq 0$.
            \end{center}
\end{Lemma}
\begin{proof}
 Previous lemma deals with the situation without repeated poins on orbits, so we work out the second situation. Using the notation (4.1), the ramification hypothesis implies that $e_i \leq d_{i_i}-1$ for each $i$. Therefore
 \begin{align*}
  e_P(\Phi^{m})=e_1e_2...e_m & \leq  \prod_{j \leq m}(d_{i_j}-1) \\ & \leq \prod_{j \leq m}\left(1 - \dfrac{1}{d_{i_j}}\right)\prod_{j \leq m} d_{i_j} \\ & \leq \left(1 - \dfrac{1}{\max_i d_i} \right)^m \prod_{j \leq m} d_{i_j} \\ &=\left(1 - \dfrac{1}{\max_i d_i} \right)^m \deg \Phi^m,
 \end{align*} which is as desired with $\kappa_1= 1, \kappa_2=\left(1 - \dfrac{1}{\max_i d_i} \right)$.

\end{proof}

\section{a bound for the number of quasiintegral points in an orbit}
In this section, we show explicit bounds for the number of $S$-integral points in a given orbit of a wandering point for a dynamical system of rational functions
extending previous work by Hsia and Silverman [3]. 

It was first showed by J. Silverman [1] that
orbits of this kind have only a finite number of $S$-integers, which we recall below.
\begin{Th}(Theorem 2.1, [7])
 Let $R_S$ be the ring of $S$-integers of $K$, and let 
 $\mathcal{F}=\{\phi_1,...,\phi_k \}$ be a set of rational functions of degree at least two defined over $K$. Let $\mathcal{O}_{\mathcal{F}}(P)$ be the orbit of $P$ under the semigroup generated by $\mathcal{F}$. Assume that no map in the semigroup is totally ramified in its fixed points. Then for any function $z
\in \mathbb{P}^1$, the set
\begin{center}
 $\{ Q \in \mathbb{P}^1(K) | Q \in \mathcal{O}_{\mathcal{F}}(P) $ and $ z(Q) \in R_S\}$
 \end{center} is finite.
\end{Th}
The next quantitative theorem generalizes a theorem of Hsia and Silverman to a semigroup situation.
\begin{Th}
 Let $\mathcal{F}= \{ \phi_1, \phi_2,...,\phi_k \} \subset K(z)$ be a set of rational maps of respective degrees $2 \leq d_1 \leq d_2 \leq ... \leq d_k$. We fix a sequence $\Phi=(\phi_{i_j})_{j=1}^\infty$ of functions in $\mathcal{F}$, 
 with $\Phi^n=\phi_{i_n}\circ...\circ \phi_{i_1} \in \mathcal{F}_n$, and $P \in \mathbb{P}^1(K)$ not preperiodic for $\Phi$. Fix $A \in \mathbb{P}^1(K)$ such that no two points in the $\Phi$-orbit of $A$ coincide, or otherwise tha no point in its orbit is totally ramified for any map in $\mathcal{F}$.
 For any finite set of places $S \subset M_K$ and any constant $1 \geq \varepsilon >0$, define a set of nonnegative integers by
\begin{center}
 $\Gamma_{\Phi,S}(A, P, \varepsilon) := \{ n \geq 0 : \sum_{v \in S} d_v \lambda_v (\Phi^n(P), A) \geq \varepsilon \hat{h}_{S^n(\Phi)}(\Phi^n (P)) \}$.
 \end{center}
 (a) There exist effective constants
 \begin{center}
  $\gamma_1= \gamma_2(d_1,...,d_k, \varepsilon, [K: \mathbb{Q}])$ and $\gamma_2= \gamma_2(d_1,...,d_k, \varepsilon, [K: \mathbb{Q}])$
 \end{center} such that
 \begin{center}
  $\# \left\{n \in \Gamma_{\Phi,S}(A, P, \varepsilon) : n >  \gamma_1 + \log_{d_1}^+ \left(\dfrac{\hat{h}_{\mathcal{F}}(A)+h(\mathcal{F})}{\hat{h}_{\Phi}(P)}\right)\right\} \leq 4^{\#S}\gamma_2$.
 \end{center}
 In particular, there is an effective constant $\gamma_3(d_1,...,d_k,\varepsilon, [K: \mathbb{Q}])$ such that 
 \begin{center}
  $\# \Gamma_{\Phi,S}(A, P, \epsilon) \leq 4^{\# S}\gamma_3 + \log_{d_1}^+ \left(\dfrac{\hat{h}_{\mathcal{F}}(A)+h(\mathcal{F})}{\hat{h}_{\Phi}(P)}\right) $
 \end{center}
 (b) If $P$ is not $\Phi$-preperiodic for each $\Phi$,  there is a constant $\gamma_3(K, S, \mathcal{F}, A, \epsilon)$ that is independent of $P$ and of the sequence $\Phi$ chosen from $\mathcal{F}$ such that
\begin{center}
 $\max_{\Phi,P} \Gamma_{\Phi,S}(A, P, \varepsilon) \leq \gamma_4$.
\end{center}

\begin{proof}
 For simplicity, we write $\Gamma_{S}(\varepsilon)$ instead of $\Gamma_{\Phi,S}(A, P, \varepsilon)$.
Taking $\kappa_1$ and $\kappa_2 < 1$ the constants from Lemmas 4.3 and 4.4, we choose $m \geq 1$ minimal such that
$\kappa^m_2 \leq \epsilon/5\kappa_1$. Then $\kappa_1, \kappa_2$ and $ m$ depend only on $d_1,.., d_k$ and on $\varepsilon$. 

If $n \leq m$ for all $n \in \Gamma_{S}(\varepsilon)$, then
\begin{center}
 $\# \Gamma_{S}(\epsilon) \leq m \leq \dfrac{\log (5\kappa_1) + \log (\varepsilon^{-1})}{\log (\kappa_2^{-1})} + 1$,
\end{center} which is in the desired form.
If there is an $n \in \Gamma_{S}(\varepsilon)$ such that $n > m$, we fix $n$ for instance. Then by definition of $\Gamma_{S}(\varepsilon)$ we have
\begin{equation}
 \epsilon \hat{h}_{S^n(\Phi)}(\Phi^n (P)) \leq  \sum_{v \in S} d_v \lambda_v (\Phi^n(P), A).
\end{equation}
We can write $\Phi^n= \psi \circ \Phi^{n-m}$ for $\psi= \phi_{i_n}\circ ... \circ \phi_{i_{n-m +1}} \in \mathcal{F}_m$.
\newline \newline For our chosen $m$, we denote
\begin{center}
 $\textbf{e}_m:= \max\limits_{A^{\prime} \in \psi^{-1}(A)} e_{A^{\prime}}(\psi)$.
\end{center} By Lemma 4.4 and our choice of $m$, we notice that
\begin{center}
 $\textbf{e}_m \leq \kappa_1 (\kappa_2)^m \deg \psi \leq \epsilon \deg \psi/5$
\end{center}

Therefore, Lemma 4.1 yields, for $Q \in \mathbb{P}^1(K)$ and $\psi \in \mathcal{F}_m$, that
\begin{equation}
 \sum_{v \in S} d_v \lambda_v ( \psi(Q),A) - O(h(A) + h(\psi) + 1)   \leq \textbf{e}_m \sum_{v\in S} \max\limits_{A^{\prime} \in \psi^{-1}(A)} d_v \lambda_v(Q, A^{\prime}).
\end{equation}
Gathering (5.1) and (5.2) with $Q:= \Phi^{n-m}(P)$, we obtain that
\begin{center}
 $\epsilon \hat{h}_{S^n(\Phi)}(\Phi^n (P)) \leq \textbf{e}_m \sum_{v\in S} \max\limits_{A^{\prime} \in \psi^{-1}(A)} d_v \lambda_v(\Phi^{n-m}(P), A^{\prime}) + O(h(A) + h(\mathcal{F}_m) + 1)$,
\end{center} where the involved constants depend only on the degree of the functions in $\mathcal{F}_m$, and so on $d_1,...d_k$ and on $\epsilon$.

For each $v \in S$, we choose $A_v^{\prime} \in \psi^{-1}(A)$ such that
\begin{center}
 $ \lambda_v(\Phi^{n-m}(P), A_v^{\prime})=\max\limits_{A^{\prime} \in \psi^{-1}(A)}  \lambda_v(\Phi^{n-m}(P), A^{\prime})$,
\end{center} so that
\begin{center}
 $\epsilon \hat{h}_{S^n(\Phi)}(\Phi^n (P)) \leq \textbf{e}_m \sum_{v\in S} d_v \lambda_v(\Phi^{n-m}(P), A_v^{\prime}) + O(h(A) + h(\mathcal{F}_m) + 1)$.
\end{center}
For instance, we can assume that $z(A^{\prime}) \neq \infty$ for all $A^{\prime} \in \Psi^{-1}(A), \Psi \in \mathcal{F}_m$. If this is not the case, we use $z$ for some of the $A^{\prime}$ and $z^{-1}$ for the others.

Let $S^{\prime} \subset S$ be the set of places in $S$ defined by
\begin{center} 
 $S^{\prime}= \{ v \in S ; \lambda_v(\Phi^{n-m}(P), A_v^{\prime}) > \lambda_v (A_v^{\prime}, \infty)+ \log l_v \}$, 
\end{center} where again $l_v=2$ for $v$ archimedean and $l_v=1$ otherwise. 

Set $S^{\prime \prime}:= S - S^{\prime}$. Applying Lemma 3.2 to the places in $S^{\prime}$ and using the definition of $S^{\prime \prime}$ we find that
\begin{align*}
 \epsilon \hat{h}_{S^n(\Phi)}(\Phi^n (P))& \leq \textbf{e}_m \displaystyle\sum\limits_{v\in S}
 d_v \lambda_v(\Phi^{n-m}(P), A_v^{\prime}) + O(h(A) + h(\mathcal{F}_m) + 1)\\
&\leq \textbf{e}_m \displaystyle\sum\limits_{v\in S^{\prime}} d_v(2 \lambda_v(A_v^{\prime},\infty)-\log|z(\Phi^{n-m}(P))- z(A_v^{\prime})|+ \log l_v)\\
& \qquad \qquad 
 +\textbf{e}_m \displaystyle\sum\limits_{v\in S^{\prime \prime}} d_v(\lambda_v(A_v^{\prime},\infty) + \log l_v) + O(h(A)+ h (\mathcal{F}_m) + 1)\\
 &\leq \textbf{e}_m \displaystyle\sum\limits_{v\in S^{\prime}} d_v \log|z(\Phi^{n-m}(P))- z(A_v^{\prime})|^{-1}\\
 & \qquad \qquad + \textbf{e}_m \displaystyle\sum\limits_{v\in S} d_v(2\lambda_v(A_v^{\prime},\infty) + \log l_v) + O(h(A)+ h (\mathcal{F}_m) + 1).
\end{align*}
Now using Lemma 2.3 and Lemma 3.5 it can be checked that
\begin{align*}
\displaystyle\sum\limits_{v\in S} d_v \lambda_v(A_v^{\prime},\infty) & \leq \displaystyle\sum\limits_{A^{\prime} \in \psi^{-1}(A)} \displaystyle\sum\limits_{v\in S} d_v \lambda_v(A^{\prime},\infty)\\
& \leq \displaystyle\sum\limits_{A^{\prime} \in \psi^{-1}(A)} h(A^{\prime})\\ &\leq \displaystyle\sum\limits_{A^{\prime} \in \psi^{-1}(A)} \hat{h}_{S^{n-m}(\Phi)}(A^{\prime}) + O(h(\mathcal{F})+1)\\ 
 &=\displaystyle\sum\limits_{A^{\prime} \in \psi^{-1}(A)} (\deg \psi)^{-1} \hat{h}_{S^m(S^{n-m}(\Phi))}(\psi(A^{\prime})) + O(h(\mathcal{F})+1)\\
&\leq \displaystyle\sum\limits_{A^{\prime} \in \psi^{-1}(A)} (\deg \psi)^{-1} \hat{h}_{S^n(\Phi)}(A) + O(h(\mathcal{F})+1)\\
&\leq \hat{h}_{S^n(\Phi)}(A) + O(h(\mathcal{F})+1)\\
&\leq \hat{h}_{\mathcal{F}}(A) + O(h(\mathcal{F}) + 1). 
\end{align*}
The constants depend only on $m$ and $d_1,...,d_k$.

Further, from the definition of $l_v$, we have 
\begin{center}
 $\sum_{v \in S}d_v \log l_v \leq \log 2$.
\end{center}
Also, from Proposition 3.3 it follows that $h(\mathcal{F}_m) = O(h(\mathcal{F}) +1)$.

All the inequalities above together imply that
\begin{center}
 $\epsilon (\hat{h}_{S^n(\Phi)}(\Phi^n (P)) \leq \textbf{e}_m (\displaystyle\sum\limits_{v\in S^{\prime}} d_v \log|z(\Phi^{n-m}(P))- z(A_v^{\prime})|^{-1})+ O(\hat{h}_{\mathcal{F}}(A) +h(\mathcal{F}) + 1)$.
\end{center}
Let us set some definitions in order to apply Roth's theorem. We define
\begin{equation}
 \Upsilon = \{z(A^{\prime}) : A^{\prime} \in \psi^{-1}(A) \} \subset \bar{K},
\end{equation} which is $G_{\bar{K} /K}$-invariant and $\# \Upsilon \leq d_k^m$. We define the map $\beta: S^{\prime} \rightarrow \Upsilon$ by $\beta_v:= A_v^{\prime}$ and analyze the points
$x = \Phi^{n-m}(P)$ for $n \in \Gamma_S(\epsilon)$. Applying Lemma 3.2 for the set of places $S^{\prime}, M=0$ and $\mu=5/2$, yields that there exist constants $r_1, r_2$ depending only on $[K: \mathbb{Q}], d_1,..., d_k$ and $\epsilon$ such that the set of $n \in \Gamma_S(\epsilon)$ with $n>m$ can be written as a union
\begin{center}
 $\{ n \in \Gamma_S(\epsilon) : n>m \}= T_1 \cup T_2 \cup T_3$
\end{center} such that \newline \newline
$ \# T_1 \leq 4^{\# S^{\prime}}r_1, \newline \newline
T_2 = \{ n >m : \displaystyle\sum\limits_{v\in S^{\prime}} d_v \log|z(\Phi^{n-m}(P))- z(A_v^{\prime})|^{-1} \leq \frac{5}{2} h(\Phi^{n-m}(P)) \}, \newline \newline
T_3 = \{ n >m : h(\Phi^{n-m}(P)) \leq r_2 \max_{v \in S^{\prime}} \{h(A_v^{\prime},1) \} \}$.
\newline \newline
We already have a bound for the size of $T_1$. For $T_3$, we use again Lemmas 2.3 and 3.5 to compute
\begin{align*}
h(A_v^{\prime}) &\leq \hat{h}_{S^{n-m}(\Phi)}(A_v^{\prime}) + c_3h(\mathcal{F}) + c_4
= (\deg \psi)^{-1}\hat{h}_{S^m(S^{n-m}(\Phi))}(A) + c_3h(\mathcal{F}) + c_4\\ 
&=(\deg \psi)^{-1}\hat{h}_{S^n(\Phi)}(A) + c_3h(\mathcal{F}) + c_4\\ 
 &\leq c_5 \hat{h}_{\mathcal{F}}(A)+ c_3h(\mathcal{F}) + c_4,
\end{align*}
 and \begin{center}
$h(\Phi^{n-m}(P)) \geq \hat{h}_{S^{n-m}(\Phi)}(\Phi^{n-m}(P)) -c_3h(\mathcal{F}) -c_4=\deg (\Phi^{n-m)})\hat{h}_{\Phi}(P)-c_3h(\mathcal{F}) -c_4$.
 \end{center}
 Hence
\begin{center}
 $T_3 \subset \{n >m : d_1^{n-m}\hat{h}_{\Phi}(P) \leq c_5 \hat{h}_{\mathcal{F}}(A)+ c_3h(\mathcal{F}) + c_4  \}$,
\end{center} so every $n \in T_3$ satisfies
\begin{center}
 $n \leq m + \log_{d_1}^+\left(\dfrac{c_5 \hat{h}_{\mathcal{F}}(A)+ c_6h(\mathcal{F}) + c_7}{\hat{h}_{\Phi}(P)}\right) \leq c_8 + \log_{d_1}^+\left(\dfrac{\hat{h}_{\mathcal{F}}(A)+h(\mathcal{F})}{\hat{h}_{\Phi}(P)}\right)$.
\end{center}
Finally, we consider the set $T_2$. Again using Lemmas 2.3 and 3.5 we derive
\begin{center}
 $h(\Phi^{n-m}(P)) \leq \hat{h}_{S^{n-m}(\Phi)}(\Phi^{n-m}(P)) + c_3 h(\mathcal{F}) + c_4\newline \newline =\deg(\Phi^{n-m})\hat{h}_{\Phi}(P)+ c_3 h(\mathcal{F}) + c_4$,
\end{center}
and then, for $n \in T_2$, using that $\textbf{e}_m \leq \epsilon \deg \psi /5$
\begin{align*}
 \epsilon \hat{h}_{S^n(\Phi)}(\Phi^n(P))&=\epsilon \deg (\Phi^n) \hat{h}_{\Phi}(P)\\
 &\leq \textbf{e}_m (\displaystyle\sum\limits_{v\in S^{\prime}} d_v \log|z(\Phi^{n-m}(P))- z(A_v^{\prime})|^{-1})+ c_9(\hat{h}_{\mathcal{F}}(A) +h(\mathcal{F}) + 1)\\
 &\leq (\epsilon \frac{\deg \psi}{5})\frac{5}{2}\deg(\Phi^{n-m})\hat{h}_{\Phi}(P)+ c_{10}(\hat{h}_{\mathcal{F}}(A)+ h(\mathcal{F}) + 1)\\
  &= \frac{\epsilon}{2}\deg(\Phi^n)\hat{h}_{\Phi}(P)+ c_{10}(\hat{h}_{\mathcal{F}}(A) +h(\mathcal{F}) + 1).
  \end{align*} Thus \newline \newline
  $\frac{\epsilon}{2}\deg(\Phi^n)\hat{h}_{\Phi}(P) \leq c_{10}(\hat{h}_{\mathcal{F}}(A) +h(\mathcal{F}) + 1)$, which implies that
  \begin{center}
   $\frac{\epsilon}{2}d_1^n \hat{h}_{\Phi}(P) \leq c_{10}(\hat{h}_{\mathcal{F}}(A) +h(\mathcal{F}) + 1)$,
  \end{center} equivalent to
  \begin{center}
   $n \leq c_{11}+ \log_{d_1}^+ \left(\dfrac{\hat{h}_{\mathcal{F}}(A)+h(\mathcal{F})}{\hat{h}_{\Phi}(P)}\right)$.
  \end{center}
  We observe that the set $\Upsilon$ defined by (5.3) does not depend on the point, so the largest element in $T_1$ is bounded  independently of $P$.
  We also note that the quantity
  \begin{center}
   $\hat{h}^{\text{min}}_{\mathcal{F},K} := \inf \{\hat{h}_{\Phi}(P) : \Phi $ a sequence of maps in $\mathcal{F}, P \in \mathbb{P}^1(K)$ is not prepriodic for $\Phi  \}$
  \end{center} is strictly positive. Namely, from Lemma 3.5, we know that
  \begin{center}
   $\hat{h}_{\Phi}(P) \leq \hat{h}_{\mathcal{F}}(P) + O(h(\mathcal{F}))$, 
  \end{center} and $O$ does not depend on $\Phi$ generated by $\mathcal{F}$, and neither on $P$.
  
  So if $P_0$ is a $\Phi$-wandering point, $J= \{ 1,..., k\}$ and $W= \prod_{i=1}^\infty J$, then $\hat{h}_{\mathcal{F}}(P)>0$ and
  \begin{center}
   $\hat{h}^{\text{min}}_{\mathcal{F},K} := \inf \{\hat{h}_{\Phi_w}(P) : w \in W, P \in \mathbb{P}^1(K)$ and $0 < \hat{h}_{\Phi_w}(P) \leq \hat{h}_{\mathcal{F}}(P) + O(h(\mathcal{F}))  \}$,
  \end{center}
for this last set is finite by the Northcott property for $\hat{h}_{\mathcal{F}}$, so the infimun is taken over a finite set of positive numbers.

Therefore, $\max(T_1 \cup T_2 \cup T_3)$ can be bounded independently of $P$ and the choice of the sequence $\Phi$ generated by the semigroup $\mathcal{F}$.
\end{proof}
Moreover, one can make the last claim of Theorem 5.2 more precise.
\end{Th}
\begin{Prop} Under the conditions and notations of the proof of Theorem 5.2, there exists $\gamma_2$ depending only on $A, \mathcal{F},K, S, \epsilon$ such that
\begin{center}
 $ \max (T_1 \cup T_2 \cup T_3) \leq \gamma_2 + \log_{d_1}^+ \left( \dfrac{\hat{h}_{\mathcal{F}}(A) + h(\mathcal{F})}{\hat{h}_{\Phi}(P)} \right)$.
\end{center}

\end{Prop}
\begin{proof}
 Due to the proof of Theorem 5.2, we only need to prove such bound for $\max T_1$.
 
 According to the proof of Theorem B from [2], page 3, inequality 6, for the algebraic numbers $x$ approximating $\alpha$ satisfying Roth's theorem hypothesis, there exists a finite number( depending on the constants given by Lemma 4.2) of $\beta_i$'s approximating $\alpha$ that  depend only on $\alpha$ and on the parameters of Lemma 4.2 such that 
 \begin{center}
  $\log(4H(x)) \leq \dfrac{4rn}{\eta}\left( \dfrac{1}{\eta} \log(4H(\alpha)) + \log (4 \max_i H(\beta_i)) \right)$,
 \end{center} where $r, n$ and $\eta$ depend only on $\# \Upsilon$ defined in (5.3).
 
 This implies that
$h(x) \leq C(h(\alpha) + \max_i h(\beta_i))$, where $C$ depends only on the parameters of Lemma 4.2.
 Translating this for the notation of our set $T_1$, as in the proof of Theorem 5.2, we have that 
 \begin{align*} 
 h(\Phi^{n-m}(P)) &\leq C(\max_{v \in S, \psi \in \mathcal{F}_m}h(A^{\prime}_v) + \max_{i,v, \psi}h(\beta_{i,v})) \\ &=O(\hat{h}_{\mathcal{F}}(A) + h(\mathcal{F}) + \max_{i,v, \psi}h(\beta_{i,v}))\\ &= O(\hat{h}_{\mathcal{F}}(A) + h(\mathcal{F})) + \gamma,
  \end{align*}
 for each $n \in T_1$,
  where $\gamma$ depends only on $A, \mathcal{F},K, S$ and $\epsilon$ by our previous choice of $m$.
 
 We have thus that
 \begin{align*}
  d_1^{n-m}\hat{h}_{\Phi}(P) & \leq \deg(\Phi^{n-m}) \hat{h}_{\Phi}(P)= \hat{h}_{S^{n-m}\Phi}(\Phi^{n-m}(P))\\ &=h(\Phi^{n-m}(P)) + O(1)\\ &\leq O(\hat{h}_{\mathcal{F}}(A) + h(\mathcal{F}))+ \gamma_1, 
  \end{align*}
  for each $n \in T_1$, where $\gamma_1$ depends only on $A, \mathcal{F},K, S$ and $\epsilon$.

 Therefore, \begin{center}
 
 $\max T_1 \leq m + \log_{d_1}^+ \left( \dfrac{O(\hat{h}_{\mathcal{F}}(A) + h(\mathcal{F}))+\gamma_1}{\hat{h}_{\Phi}(P)}  \right) \leq \gamma_2 + \log_{d_1}^+ \left( \dfrac{\hat{h}_{\mathcal{F}}(A) + h(\mathcal{F})}{\hat{h}_{\Phi}(P)} \right)$,
\end{center}where $\gamma_2(\{ \beta_{i,v})\}_{i,v, \psi})$ depends only on $A, \mathcal{F},K, S$ and $\epsilon$, concluding the proof.
\end{proof}
\begin{Cor}
 Let $S \subset M_K$ be a finite set of places that includes all archimedean places, let $R_S$ be the ring of $S$-integers of $K$, and let $2\leq d_1 \leq ... \leq d_k.$ There is an effective constant
 $\gamma= \gamma(d_1,...,d_k, [K:\mathbb{Q}])$ such that for all sets $\mathcal{F}= \{ \phi_1,..., \phi_k \} \subset K(z)$ of $k$ rational functions of respective degrees $d_1,...,d_k$ that are not totally ramified at the $\mathcal{F}$-orbit of $\infty$ or that the $\mathcal{F}$-orbit of $\infty$ has no repeated points, and for any sequence $\Phi$ of maps from $\mathcal{F}$ and all points $P \in \mathbb{P}^1(K)$ that are not prepriodic for $\Phi$,
 the number of $S$-integers in the $\Phi-$orbit of $P$ is bounded by
 \begin{center}
  $\# \{n \geq 1 | z(\Phi^n(P)) \in R_S \} \leq 4^{\#S}\gamma + \log^+_{d_1}\left(\dfrac{h(\mathcal{F})}{\hat{h}_{\Phi}(P)}\right)$.
 \end{center}

\end{Cor}
\begin{proof}
An element $\alpha \in K$ is in $R_S$ if and only if $|\alpha|_v \leq 1$ for all $v \not\in S$, or equivalently, if and only if
\begin{center}
 $h(\alpha) = \sum_{v \in S}d_v \log \max \{ |\alpha|_v,1 \}$.
\end{center} Another fact is that
\begin{center}
 $\log \max \{ |\alpha|_v,1 \} \leq \lambda_v(\alpha, \infty)$.
\end{center} This implies for $\alpha \in R_S$ that $h(\alpha) \leq \sum_{v \in S} d_v \lambda_v(\alpha, \infty)$.

Let $ n \geq 1$ satisfy $z(\Phi^n(P)) \in R_S$. Then
\begin{center}
 $h(\Phi^n(P)) \leq \sum_{v \in S} d_v \lambda_v(\Phi^n(P), \infty)$.
\end{center}
Lemmas 3.5 and 2.3 tell us that
\begin{center}
 $h(\Phi^n(P)) \geq \hat{h}_{S^n(\Phi)}(\Phi^n(P)) - c_3h(\mathcal{F}) - c_4=\deg(\Phi^n)\hat{h}_{\Phi}(P) - c_3h(\mathcal{F}) - c_4$,
\end{center} which implies that
\begin{center}
 $\deg(\Phi^n)\hat{h}_{\Phi}(P) - c_3h(\mathcal{F}) - c_4 \leq \sum_{v \in S} d_v \lambda_v(\Phi^n(P), \infty)$.
\end{center}
The rest of  the proof is divided in two cases:
First one, when 
\begin{center}
 $\deg(\Phi^n)\hat{h}_{\Phi}(P) \leq 2c_3h(\mathcal{F}) + 2 c_4$.                
\end{center}In this case, $d_1^n \hat{h}_{\Phi}(P) \leq 2C_3h(\mathcal{F}) + 2 c_4$, and then
\begin{center}
 $n \leq \log_{d_1}^+ \left(  \dfrac{2c_3h(\mathcal{F}) + 2c_4}{\hat{h}_{\Phi}(P)}  \right)$.
\end{center} In the second case , $\deg(\Phi^n)\hat{h}_{\Phi}(P) \geq 2c_3h(\mathcal{F}) + 2 c_4$. Therefore
\begin{center}
 $\sum_{v \in S} d_v \lambda_v(\Phi^n(P), \infty) \geq \frac{1}{2}\deg(\Phi^n)\hat{h}_{\Phi}(P)=\frac{1}{2}\hat{h}_{S^n(\Phi)}(\Phi^n(P))$.
\end{center} Now the previous theorem with $\epsilon =1/2, A= \infty$ ($\infty$ is not totally ramified for any map of the system) tells us that $n$ is at most
\begin{center}
 $4^{\#S} \gamma_3 + \log^+_{d_1} \left( \dfrac{h(\mathcal{F}) + \hat{h}_{\Phi}(\infty)}{\hat{h}_{\Phi}(P)} \right)$,
\end{center} for $\gamma_3$ depending only on $[K: \mathbb{Q}], d_1,...,d_k$. Both bounds are on the desired form since $\hat{h}_{\Phi}(\infty) \ll h(\infty)=0$.
\end{proof}

\begin{Cor}
Under the conditions of Corollary 5.4, there is a constant
 $\gamma= \gamma(S,\mathcal{F}, [K:\mathbb{Q}])$ such that for all sets $\mathcal{F}= \{ \phi_1,..., \phi_k \} \subset K(z)$ of rational functions of respective degrees $d_1,...,d_k$ that are not totally ramified at the $\mathcal{F}$-orbit of $\infty$ or that the $\mathcal{F}$-orbit of $\infty$ has no repeated points, and all points $P \in \mathbb{P}^1(K)$ that are not preperiodic for any sequence $\Phi$ of terms in $\mathcal{F}$,
 the number of $S$-integers in the $\mathcal{F}-$orbit of $P$ is bounded by
 \begin{center}
  $\# \{Q \in \mathcal{O}_{\mathcal{F}}(P) | z(Q) \in R_S \} \leq \dfrac{k^M-1}{k-1}$, 
 \end{center} where \begin{center}$M=\left\lceil \gamma + \log^+_{d_1}\left(\dfrac{h(\mathcal{F})}{\hat{h}^{\text{min}}_{\mathcal{F},K}(P)}\right)\right\rceil+1$. \end{center}
\end{Cor} 
\begin{proof}
 If $Q \in \mathcal{O}_{\mathcal{F}}(P), z(Q) \in R_S$, then there exists a sequence $\Phi$ of maps from $\mathcal{F}$, and a $n \geq 1$, such that $Q = \Phi^n(P)$ and $z(\Phi^n(P)) \in R_S$. By Theorem 5.2, Corollary 5.4 and Proposition 5.3,
 there exists a suitable $\gamma$ such that \begin{center} $n \leq \gamma + \log^+_{d_1}\left(\dfrac{h(\mathcal{F})}{\hat{h}^{\text{min}}_{\mathcal{F},K}(P)}\right)$.
\end{center} And for each $m$, there are at most $k^m$ maps inside $\mathcal{F}_m$, and therefore at most $k^m$  $S$-integer points on the set $\{ f(P)| f \in \mathcal{F}_m \}$. The result follows from the identity
$1 + k +...+ k^n = \dfrac{k^{n+1}-1}{k-1}$.
 \end{proof}
 \begin{Rem}
  In the particular case of a system of polynomial maps, that are non-special (not monomial and neither Tchebychev's), the number of points for which some $\Phi$-orbit has repeated points if finite, due to Theorem 1.7 of [6], and therefore only for a finite number of points $A$ the hypothesis of Theorem 5.2 will not be satisfied.
 \end{Rem}
 
 \begin{Rem}
 Theorem 5.2 delivers, in particular, under its conditions for sequences $\Phi$ of rational functions in a given system over a certain number field and $P,A$ rational numbers, an explicit upper bound for
 \begin{center}
 $\# \{ n \geq 1 ; \dfrac{1}{\Phi^n(P) - A} $ is quasi-$(S, \epsilon)$-integral $ \}$,
  
 \end{center} and this does not depend on which $\Phi$ was chosen from the initial set $\mathcal{F}$.

 \end{Rem}
 
 \begin{Cor}
  Under the hypothesis of Theorem 5.2,
  \begin{center}
   $\displaystyle\lim_{n \rightarrow \infty} \dfrac{\lambda_v(\Phi^n(P),A)}{\deg (\Phi^n)}=0$ for every $v \in M_K$.
  \end{center}
\end{Cor}
\begin{proof}
 Applying theorem 5.2 for the set of places that contains just the place $v$, we conclude that for every natural $n$ big enough, it will be true that 
 \begin{center}
  $\dfrac{\lambda_v(\Phi^n(P))}{\deg (\Phi^n)} \leq \epsilon \dfrac{\hat{h}_{\Phi}(P)}{d_v}$.
 \end{center} Choosing $\epsilon$ sufficiently small, the result is proven.

\end{proof} Note that, due to Theorem 5.2, the convergence above has an uniformity for the semigroup of maps, in the sense that the big natural $n$ does not depend on the $\Phi$ chosen in the semigroup generated by the initial dynamical system, so that actually the stronger fact
\begin{center}
 $\displaystyle\lim_{n \rightarrow \infty} \left( \sup_{\Phi \text{ seq. of } \mathcal{F}}\dfrac{\lambda_v(\Phi^n(P),A)}{\deg (\Phi^n)} \right)=0$ for every $v \in M_K$
\end{center} is also true.
\begin{Cor}
 Suppose that a set $\mathcal= \{ \phi_1,..., \phi_k \} \subset \mathbb{Q}(z)$ of rational functions of degree at least $2$ satisfies the hypothesis of Theorem 5.2 with $P= \alpha \in \mathbb{Q}$, $A= 0$ and $A= \infty$, and let $\Phi$ be a sequence of functions of $\mathcal{F}$ such that $\# \mathcal{O}_{\Phi}(\alpha)= \infty$. Write
 \begin{center}
  $\Phi^n(\alpha)= \dfrac{a_n(\alpha)}{b_n(\alpha)} \in \mathbb{Q}$ as a fraction in lowest terms.
 \end{center} Then
 \begin{center}
  $\displaystyle\lim_{n \rightarrow \infty} \dfrac{\log |a_n(\alpha)|}{\log |b_n(\alpha)|}=1$.
 \end{center}
\end{Cor}
\begin{proof}
 From previous corollary, for $v$ the place at infinity, it is true that
 \begin{center}
  $\displaystyle\lim_{n \rightarrow \infty} \dfrac{\lambda_v(\Phi^n(\alpha),0)}{\deg (\Phi^n)}=\displaystyle\lim_{n \rightarrow \infty} \dfrac{\lambda_v(\Phi^n(\alpha),\infty)}{\deg (\Phi^n)}=0$.
 \end{center} Working out similarly as in the proof of previous Corollary, using Lemma 3.4 (i), it is true that 
\begin{center}
 $\displaystyle\lim_{n \rightarrow \infty} \dfrac{\lambda_v(\Phi^n(\alpha),0)}{ h(\Phi^n(\alpha))}=\displaystyle\lim_{n \rightarrow \infty} \dfrac{\lambda_v(\Phi^n(\alpha),\infty)}{h(\Phi^n(\alpha))}=0$.
\end{center} 
On the other hand, if $t= \dfrac{a}{b} \in \mathbb{Q}$ written in lowest terms, since $\max \{ |a|, |b| \} \leq \sqrt{|a|^2 + |b|^2}$, then $h(t)= \log \max \{ |a|,|b| \}$ and
\begin{align*}\lambda_v(t, \infty) &= \lambda_v([a,b],[1,0])= \log \left( \dfrac{\sqrt{|a|^2 + |b|^2}}{|b|} \right)= -\log|b| + \log(\sqrt{|a|^2 + |b|^2})\\ & \geq -\log |b| + h(t).
 \end{align*} And in the same way
 \begin{align*}\lambda_v(t, 0) &= \lambda_v([a,b],[0,1])= \log \left( \dfrac{\sqrt{|a|^2 + |b|^2}}{|a|} \right)= -\log|a| + \log(\sqrt{|a|^2 + |b|^2})\\ & \geq -\log |a| + h(t).
 \end{align*}
 Gathering these facts, and recalling that $\Phi^n(\alpha)= \dfrac{a_n(\alpha)}{b_n(\alpha)}$  yields
 \begin{center}
  $\displaystyle\lim_{n \rightarrow \infty} \dfrac{-\log|b_n(\alpha)| + h(\Phi^n(\alpha))}{ h(\Phi^n(\alpha))}=\displaystyle\lim_{n \rightarrow \infty} \dfrac{-\log|a_n(\alpha)| + h(\Phi^n(\alpha))}{h(\Phi^n(\alpha))}=0$,
 \end{center} and thus
 
\begin{center}
  $\displaystyle\lim_{n \rightarrow \infty} \dfrac{\log|b_n(\alpha)|}{ \log \max \{ |a_n(\alpha)|, |B_n(\alpha)| \}}=\displaystyle\lim_{n \rightarrow \infty} \dfrac{\log|a_n(\alpha)|}{\log \max \{ |a_n(\alpha)|, |B_n(\alpha)| \}}=1$.
 \end{center} This implies that 
 \begin{center}
 $\displaystyle\lim_{n \rightarrow \infty} \dfrac{\log |a_n(\alpha)|}{\log |b_n(\alpha)|}=1$.

\end{center}

\end{proof}
 \begin{Rem}Again, from Theorem 5.2, for a given $\alpha \in \mathbb{Q}$, the last result does not depend on $\Phi$, in the sense that for every sequence $\Phi$ of functions in the tree of functions  belonging the initial set $\mathcal{F}$, the correspondent quotient sequences $\dfrac{\log |a_n(\alpha)|}{\log |b_n(\alpha)|}$ converge to $1$ as $n$ goes to $\infty$ with the same speed. 
\end{Rem} The $\Phi$-uniformity pointed in Remark 5.10 results in the following semigroup integrality result.
\begin{Cor}
 Suppose that a set $\mathcal= \{ \phi_1,..., \phi_k \} \subset \mathbb{Q}(z)$ of rational functions of degree at least $2$ satisfies the hypothesis of Theorem 5.2 with $P= \alpha \in \mathbb{Q}$, $A= 0$ and $A= \infty$, and that $\# \mathcal{O}_{\Phi}(\alpha)= \infty$ for each sequence $\Phi$ of functions of $\mathcal{F}$. Suppose also that for some $M >0$, $\Phi^n(\alpha) \neq 0$ for any $n >M$ and any $\Phi$ sequence of maps in $\mathcal{F}$ Write
 \begin{center}
  $f(\alpha)= \dfrac{a_f(\alpha)}{b_f(\alpha)} \in \mathbb{Q}$ as a fraction in lowest terms
 \end{center} for each $f$ in the semigroup generated by $\mathcal{F}$.
 
 Then
 \begin{center}
  $\displaystyle\lim_{n \rightarrow \infty} \dfrac{1}{k^n}\left\{\displaystyle\sum_{f \in \mathcal{F}_n} \dfrac{\log |a_f(\alpha)|}{\log |b_f(\alpha)|}\right\}=1$.
 \end{center}
\end{Cor}
\begin{proof}
 By Corollary 5.9 with its notation, for any $\varepsilon>0$ we have that \begin{center} $1- \varepsilon \leq \dfrac{\log |a_n(\alpha)|}{\log |b_n(\alpha)|} \leq 1+ \varepsilon$ \end{center} for $n$ sufficiently large, and uniformly for $\Phi$. For such  numbers $n$, this implies that
 \begin{center}
  $1- \varepsilon \leq \dfrac{1}{k^n}\left\{\displaystyle\sum_{f \in \mathcal{F}_n} \dfrac{\log |a_f(\alpha)|}{\log |b_f(\alpha)|}\right\} \leq 1+ \varepsilon$,
 \end{center} from where the result follows, since $\varepsilon$ is arbitrary.

\end{proof}

\textbf{Acknowledgements}: I would like to especially thank Alina Ostafe, John Roberts and Igor Shparlinski for encouraging me to this think in this problem, and 
I am very grateful for the much helpful suggestions of Igor Shparlinski and Alina Ostafe, and every clarifying conversation that we had about this research.
I am very thankful to ARC Discovery Grant and UNSW for supporting me in this research.

\end{document}